\documentclass[11pt,letterpaper]{amsart}
\usepackage[utf8]{inputenc}
\usepackage[english]{babel}
\usepackage[T1]{fontenc}
\usepackage{lmodern}
\usepackage{amsthm}
\usepackage{amsmath}
\usepackage{faktor}
\usepackage{hyperref}
\usepackage{relsize}
\usepackage[left=2.8cm,right=2.8cm,top=3cm,bottom=2cm]{geometry}
\usepackage{amssymb}
\usepackage{tikz} 
\usepackage{enumitem}
\usepackage{varioref}
\usetikzlibrary{knots}
\usetikzlibrary{decorations.markings,hobby,knots,celtic}
\definecolor{lstbgcolor}{rgb}{0.9,0.9,0.9}
\usepackage{bbm}
\newtheorem{theorem}{Theorem}[section]
\newtheorem{lemma}[theorem]{Lemma}
\newtheorem{remark}[theorem]{Remark}
\newtheorem{definition}[theorem]{Definition}
\newtheorem{prop}[theorem]{Proposition}
\newtheorem{corollary}[theorem]{Corollary}
\newtheorem{Conj}[theorem]{Conjecture}

\title{Dimension of the skein module of a Dehn filling}
\date{}
\author{EDWIN KITAEFF}

\begin{document}

\begin{abstract}
Given a knot $K$ and a generic slope $r$, we study the Kauffman bracket skein module (KBSM) $S(E_K (r) , \mathbb{Q} (A))$ of the Dehn filling $E_K (r)$ of slope $r$ along $K$, assuming that the KBSM $S(E_K , \mathbb{Q} [A^{\pm 1}])$ of the exterior $E_K$ of $K$ is finitely generated over $S(\partial E_K ,\mathbb{Q} [A^{\pm 1}])$.  As shown in \cite{Le06}, this condition is satisfied for $K$ a two-bridge knot. In this setting, we show that $\dim_{\mathbb{C}} (S_\zeta (E_K (r))) = \dim_{\mathbb{Q} (A)} (S (E_K (r)))$ for almost all primitive roots of unity $\zeta$ of order $2N$ with $N$ odd, and for almost all slopes $r$. When the character variety of a 3-manifold $M$ is finite, we also discuss the decomposition of $S_\zeta (M)$ in terms of localized skein modules. In particular, the dimension of the localized skein modules at a non-central point is the multiplicity of this point.
\end{abstract}
\maketitle
\section{Introduction}
\subsection{The Kauffman bracket skein module}
Let $M$ be a compact oriented 3-manifold, let $R$ be a commutative ring and let $A$ be a choice of an invertible element of $R$. The Kauffman bracket skein module $S_A (M,R)$, or simply skein module here, was introduced independently by Przytycki (\cite{Prz}) and Turaev (\cite{Turaev}). It is defined as the $R$-module spanned by the framed links in $M$ modulo isotopies and the Kauffman skein relations :
\begin{center}
\includegraphics[scale=0.3]{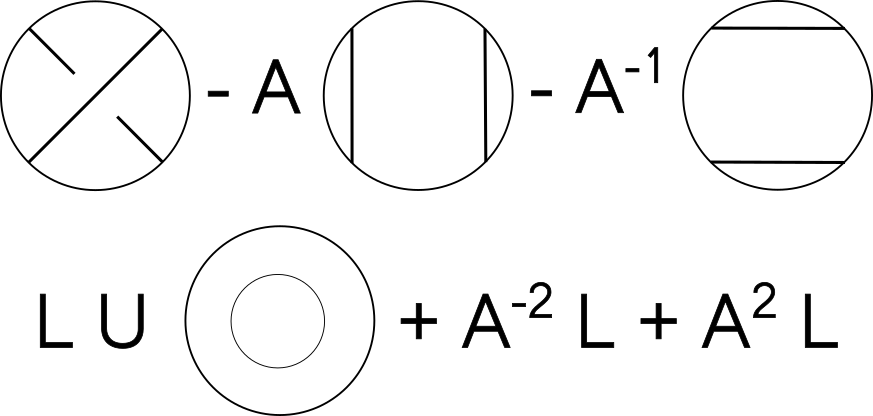}
\end{center}
For a surface $\Sigma$, we write $S_A (\Sigma ,R)$ instead of $S_A (\Sigma \times I ,R)$.\\
Furthermore, if the choice of $A$ is clear, we simply write $S(M,R)$, or even $S(M)$ if $R=\mathbb{Q} (A)$. For $\zeta\in\mathbb{C}^*$, we define $S_\zeta (M) := S_\zeta (M,\mathbb{C})$.

Although the definition of the skein module is quite simple, its computation is notoriously difficult. In fact, the skein module 
$S(M,\mathbb{Q} [A^{\pm 1}])$ is known only for a limited number of 3-manifolds, such as lens spaces \cite[Theorem 4]{HosPrz93}, $\mathbb{S}^2 \times \mathbb{S}^1$\cite{HosPrz95}, the exterior of a 2-bridge knot \cite[Theorem 2]{Le06}, $\mathbb{R}P^3 \# \mathbb{R}P^3$ \cite[Theorem 1]{Mro} and $(\mathbb{S}^2\times \mathbb{S}^1) \# (\mathbb{S}^2\times \mathbb{S}^1)$ \cite[\textsection 3]{BKSW}.\\
Regarding the skein module with coefficient in $\mathbb{Q} (A)$, we don't know many more examples of computations. However, the following striking result tells us about the structure of $S(M)$ :

\begin{theorem}\cite[Theorem 1]{GJS}\label{Thm:GJS}
For $M$ a closed 3-manifold, $S(M)$ is a finite dimensional $\mathbb{Q} (A)$-vector space.
\end{theorem}
Unfortunately, the proof of \cite{GJS} is not constructive and cannot be used to compute $S(M)$.\\ 
An alternative proof can be found in \cite{BelDet}, where, unlike \cite{GJS}, the dimension of $S(M)$ is bounded (from above) using an algorithm that computes an explicit set of generators. However, this set is often not optimal.\smallbreak
On the other hand, the interpretation of the dimension of $S(M)$ is the subject of the following conjecture :
\begin{Conj}\cite[Section 6.3]{GJS}\label{Conj:Instanton-Floer}
For a closed 3-manifold $M$, we have that $$\dim_{\mathbb{Q} (A)} (S(M))= \dim_{\mathbb{C} } HP^0_\# (M)$$ where $HP_\# (M)$ is the Abouzaid-Manolescu homology \cite{AboCip}.
\end{Conj}
Overall, the computation and interpretation of the dimension of $S(M)$ remain a very difficult and open problem.

\subsection{Setting}
The skein module $S_A (\Sigma , R)$ has an algebraic structure induced by the operation $\alpha \star \beta$ of stacking $\alpha$ over $\beta$. For a 3-manifold $M$ with boundary, this gives $S_A (M,R)$ the structure of an $S_A (\partial M ,R)$-module.

Let $K$ be a knot, $E_K$ be the knot complement and $E_K (r)$ the surgery on $K$ of slope $r$.\\
As a skein module of a manifold with boundary, the skein module $S(E_K,\mathbb{Q} [A^{\pm 1}])$ has the structure of an $S(\mathbb{T}^2,\mathbb{Q} [A^{\pm 1}])$-module.

The condition we require for $K$ is the following :
\begin{equation}\label{Condition}\tag{$\star$}
S(E_K ,\mathbb{Q} [A^{\pm 1}])\textrm{ is finitely generated over }S(\partial E_K ,\mathbb{Q} [A^{\pm 1}]).
\end{equation}
\begin{theorem}\cite[Theorem 2]{Le06}\label{Thm:TwoBridgeKnots1}
Any two-bridge knot satisfies (\ref{Condition}).
\end{theorem}
We remark that condition (\ref{Condition}) still make sense when $E_K$ is replaced by a 3-manifold $M$ with $\partial M \simeq \mathbb{T}^2$ and our main result (Theorem \ref{Thm:Main}) applies to that setting as well. Since our main example comes from surgeries on 2-bridge knots, we continue in this setting.

\subsection{The results}
It is well known that the skein module is deeply connected with the character variety. We use this connection to compute the dimension of $S(M)$.

For an oriented connected manifold $M$, let 
$$
\chi (M) = Hom(\pi_1 (M),SL_2 (\mathbb{C}))/\!\!/ SL_2 (\mathbb{C} )
$$
be the $SL_2 (\mathbb{C})$-character scheme of $M$ and $X (M)$ its underlying algebraic set. The character variety $X(M)$ can be seen as the quotient of $Hom(\pi_1 (M),SL_2 (\mathbb{C}))$ in which two representations are identified if and only if their traces coincide. If $\mathbb{C} [\chi (M)]$ has no non-trivial nilpotent elements we say that the character variety is reduced. It is important to note that, when $X(M)$ is finite and reduced, Conjecture \ref{Conj:Instanton-Floer} becomes that the dimension of $S(M)$ is the number of characters of $X(M)$.

Here is how the character variety and the skein module are related :

\begin{theorem}\cite{Bul}\cite{PrzSik}\label{Thm:BPS}
$$S_{-1} (M) \simeq \mathbb{C} [\chi (M)]$$
\end{theorem}
This result was initially established in \cite{Bul} up to nilpotents, and later fully proven in \cite{PrzSik}. The isomorphism of Theorem \ref{Thm:BPS} associates to a link $L$ with components $K_1,\ldots ,K_n$ the element $\left( \rho \rightarrow \prod\limits_{i=1}^{n} (-tr (K_i ))\right) \in \mathbb{C} [\chi (M)]$.

This connection was exploited in \cite{DKS25} under a property called tameness by its authors :\\
We say that $S (M, \mathbb{Q} [A^{\pm 1}])$ is tame if it can be expressed as a direct sum of cyclic $\mathbb{Q} [A^{\pm 1}]$-modules and does not contain $\mathbb{Q} [A^{\pm 1}]/(\phi_{2N} )$ as a submodule for at least one odd N, where $\phi_{2N}$ is the $2N$-th cyclotomic polynomial.

The main result of \cite{DKS25} is the following :
\begin{theorem}\cite[Theorem 1.1]{DKS25}\label{Thm:DetKalSik} 
Let $M$ be a closed 3-manifold such that $S(M,\mathbb{Q} [A^{\pm 1}] )$ is tame and $X(M)$ is finite.\\
Then, for almost all primitive $2N$-roots of unity $\zeta$, $$\dim_{\mathbb{Q} (A)} S(M) = \dim_{\mathbb{C}} S_\zeta (M) = \vert X(M) \vert$$
where $\vert X(M) \vert$ is the number of points in $X(M)$ counted with multiplicity.
\end{theorem}
\begin{remark}\label{Rk:NotReduced}
Originally, the hypothesis of \cite{DKS25} on $M$ includes the fact that $X(M)$ has to be reduced. However, one can use the work of \cite{TehFroKan} to remove this condition. We make this more precise in Section \ref{Subsec:Comp}.
\end{remark}
However, the tameness condition is not easy to check. The 3-manifolds that we know tame usually satisfy the stronger condition of having $S(M,\mathbb{Z}[A^{\pm 1}])$ finitely generated over $\mathbb{Z} [A^{\pm 1}]$ and are essentially lens spaces \cite[Theorem 4]{HosPrz93}, Dehn fillings on the figure-eight knot \cite[Theorem 4.3]{DKS25}, Dehn fillings on $(2,2n+1)$-torus knots \cite[Theorem 4.3]{DKS25} and small Seifert manifolds \cite[Theorem. 1.2]{DKS24}. Nevertheless, it is conjecture in \cite[Conjecture 1.1]{DKS24} that every small 3-manifolds is tame.

In this paper, we prove the left part of Theorem \ref{Thm:DetKalSik} for Dehn fillings satisfying (\ref{Condition}) without the tameness condition (Proof in Section \ref{SubSec:Proof}) :
\begin{theorem}\label{Thm:Main}
Under condition (\ref{Condition}), for almost all slopes $r$ and almost all primitive $2N$-roots of unity,
$$
dim_{\mathbb{Q} (A)} S(E_K (r)) = \dim_{\mathbb{C}} S_\zeta (E_K (r))
$$
\end{theorem}
Because of \cite[Theorem 2.1]{DKS25} : $\dim_{\mathbb{C}} (S_\zeta  (E_K (r) ))\geq \vert X(E_K (r)) \vert$ holds for infinitely many $\zeta$, then we have :
\begin{corollary}\label{Cor:FiniteX}
Under condition (\ref{Condition}), for almost all slopes $r$, $X(E_K (r))$ is finite.
\end{corollary}

\subsection{Character variety, reduced skein module and localized skein module}

The structure of $S_\zeta (M)$ is also deeply connected to the character variety through the threading map of \cite{BohWon} that we recall below. First, we need to redefine Chebychev polynomials of the first kind:

\begin{equation}\label{Chebychev1}\tag{T}
\left\lbrace \begin{array}{lrcl}
T_0 = 2 , \ T_1 = X  \\ 
\forall n\geq 2 , \ T_n = XT_{n-1} - T_{n-2}\end{array}\right.
\end{equation}

Let $\zeta$ be a primitive $2N$-root of unity with $N$ odd. It was shown in \cite{Le15} that for a link $L$, the element $T_N (L) \sqcup L'$ of $S_\zeta (M)$ only depends on $L'$ and on the homotopy class of $T_N (L)$. Let $\tau : S_{-1} (M) \rightarrow S_{\zeta } (M)$ be the linear map defined by $\tau (L) = T_N (L)$.\\
Then,
\begin{theorem}\cite{BohWon}\cite{Le15}\label{Thm:BohWon}
For $\zeta$ a primitive $2N$-root of unity with $N$ odd, $\tau$ gives to $S_\zeta (M)$ a structure of $S_{-1} (M)$-module.
\end{theorem}
In the end, $S_\zeta (M)$ has a structure of $\mathbb{C} [\chi (M)]$-module.\\
As an affine variety, the maximal ideals of $\mathbb{C} [\chi (M)]$ correspond to the points of $\chi (M)$, we denote $MaxSpec(S_{-1} (M)) = \{ \mathfrak{m}_{[\rho]} ,[\rho] \in \chi (M) \}$.\\
Following \cite{TehFroKan}, we define the reduced skein module at a character $[\rho ]\in\chi  (M)$ to be :
$$S_{\zeta,[\rho ]} (M) :=S_\zeta (M)  \underset{S_{-1} (M)}\bigotimes \faktor{S_{-1} (M)}{\mathfrak{m}_{[\rho ]} } $$
And the localized skein module at $[\rho ]$ is :
$$ S_{\zeta} (M)_{[\rho ]} =  S_\zeta (M) \underset{S_{-1} (M)}\bigotimes (S_{-1} (M) \setminus \mathfrak{m}_{[\rho ]})^{-1}  S_{-1} (M)$$

When $[\rho ]$ is an isolated and reduced point of $S_{-1} (M)$, we have that $S_{-1,[\rho]} (M) \simeq S_{-1} (M)_{[\rho]}$ and the localized skein module at $[\rho ]$ has the same dimension over $\mathbb{C}$ as the reduced skein module at $[\rho ]$. 

Let $L,L'$ be links in $M$ and $K_1 ,\ldots K_n$ be the components of $L$. Because of the structures given by Theorem \ref{Thm:BPS} and Theorem \ref{Thm:BohWon}, the following relation holds in $S_{\zeta,[\rho ]} (M)$ : $ T_N (L) \sqcup L' = (\prod\limits_{i=1}^{n} -tr(\rho (K_i) ) ) L$. In fact $S_{\zeta,[\rho ]} (M)$ is the quotient of $S_\zeta (M)$ by all relations of this type.

When $\chi (M)$ is finite, $S_{-1} (M)$ is Artinian, we then have the following decomposition :
$$ S_\zeta (M) = \underset{[\rho]\in\chi (M)}{\bigoplus} S_{\zeta} (M)_{[\rho ]}$$

In Section \ref{Sec:Main}, we show Theorem \ref{Thm:Main}. To do so, we adapt a proof of \cite{Det_Ek} to find, under condition (\ref{Condition}), a finitely generated localisation of $S(E_K (r) ,\mathbb{Q} [A^{\pm 1}] )$. After which, we follow a line of reasoning presented in \cite{DKS25} to show that the free part of this localisation has the same rank as $S(E_K (r))$ and $S_\zeta (E_K(r))$ for almost all roots of unity $\zeta$ of order $ord(\zeta )\equiv 2 \pmod{4}$.

In Section \ref{Sec:CharVar} we discuss the right part of Theorem \ref{Thm:DetKalSik} in the case where $X(M)$ finite, not necessarily reduced, and without the tameness condition.
\subsection{Acknowledgement}
I would like to thank my PhD supervisor, Renaud Detcherry, for his substantial help during the elaboration of this paper. I would also like to thank Victor Chachay for helping me understand the algebraic geometry background required for this work.
\section{Theorem \ref{Thm:Main}}\label{Sec:Main}
\subsection{A finitely generated localisation of \texorpdfstring{$S(E_K (r) ,\mathbb{Q} [A^{\pm 1}] )$}{}}\label{SubSec:fgloc}
For a polynomial $U \in \mathbb{Q} [A^{\pm 1}]$, denote $R_U := \mathbb{Q} [A^{\pm 1}] [U^{-1}]$.\\
The main result of this section will be Proposition \ref{Prop:Tame}. However, it needs some technicalities to be stated in its precise form. Still, we give a paraphrase here :
\begin{corollary}\label{Cor:Tame}
For $K$ verifying condition (\ref{Condition}), there exists a polynomial $U$ such that for almost all slopes $r$, $S(E_K (r) , R_U )$ is finitely generated over $R_U$.
\end{corollary}
The main tool here is the Frohman-Gelca basis of $S(\mathbb{T}^2 ,\mathbb{Q} [A^{ \pm 1}])$ used on $S(E_K,\mathbb{Q} [A^{\pm 1}])$ through its $S(\partial E_K ,\mathbb{Q} [A^{\pm 1}] ) \simeq S(\mathbb{T}^2 , \mathbb{Q} [A^{\pm 1}])$-module structure. We describe the Frohman-Gelca basis below.

Fixing two oriented curves $\lambda$ and $\mu$ intersecting once on $\mathbb{T}^2$, let $x,y$ be coprime integers, we define $\gamma_{(x,y)}$ to be the skein element represented by an oriented curve of homology class $x\lambda  + y\mu $ on $\mathbb{T}^2 \times I$. In our context, we choose $\lambda$ to be a meridian of $K$ and $\mu$ a longitude. The multicurves $\gamma_{(x,y)}^n$, consisting of $n$ parallel copies of $\gamma_{(x,y)}$, together with the empty curve, form a basis of $S(\mathbb{T}^2, \mathbb{Q} [A^{\pm 1}])$.

Recall that the definition of Chebychev polynomials of the first kind $\{ T_n \}$ is given at (\ref{Chebychev1}).\\
Frohman and Gelca introduced the following basis of $S(\mathbb{T}^2  , \mathbb{Q} [A^{\pm 1}])$, for which the product (stacking operation) satisfies the so-called product-to-sum formula :

\begin{theorem}\label{Thm:FormulaT}\cite[Theorem 1]{FroGel}
The family $\{ (x,y)_T := T_{d} (\gamma_{(\frac{x}{d} , \frac{y}{d})} ) ,\ d=gcd(x,y) \}$ is a basis for $S(\mathbb{T}^2, \mathbb{Q} [A^{\pm 1}])$ for which we have the following :
$$
(x,y)_T \star (z,t)_T = A^{xt-yz} (x+z,y+t)_T + A^{yz-xt} (x-z,y-t)_T
$$
\end{theorem}
\begin{remark}
Here, we choose the convention $(0,0)_T  = 2\cdot \emptyset$.
\end{remark}

The proof of Proposition \ref{Prop:Tame} is similar to that in \cite{Det_Ek}, and starts with the following Lemma.

\begin{lemma}\label{Lemma:Polygone}
For any knot $K'$ and for every $f \in S(E_{K'} , \mathbb{Q} [A^{\pm 1}] )$, there exists a polygon $\mathcal{P}^f$ with vertices in $\mathbb{Z}^2$ and coefficients $c_{\alpha , \beta }^f \in \mathbb{Q} [A^{\pm 1}]$ such that,
$$
\left(\mathlarger{\mathlarger{\sum}}\limits_{(\alpha, \beta) \in \mathcal{P}^f \cap \mathbb{Z}^2 } c_{\alpha ,\beta }^f (\alpha ,\beta )_T \right) \cdot f =0
$$
Where $(\alpha , \beta )_T \in S(\partial E_{K'} , \mathbb{Q} [A^{\pm 1}]) \simeq S( \mathbb{T}^2, \mathbb{Q} [A^{\pm 1}])$.\\
Moreover, the coefficients can be chosen so that $\left\lbrace \begin{array}{lrcl}
(-\mathcal{P}^f)=\mathcal{P}^f\\
\forall (\alpha ,\beta ) \in \mathcal{P}^f \cap \mathbb{Z}^2 ,\ c_{\alpha,\beta}^f=c_{-\alpha ,-\beta}^f   \\ 
\forall (\alpha ,\beta ) \in \partial \mathcal{P}^f \cap \mathbb{Z}^2 ,\ c_{\alpha , \beta }^f \neq 0 \end{array}\right.
$.
\end{lemma}

\begin{proof}
It is stated in \cite[Corollary 1.7]{BelDet} that as long as the boundary of a compact oriented 3-manifold $M$ is not a disjoint union of spheres, we have that for every $f\in S(M,\mathbb{Z} [A^{\pm 1}])$, there exists a non-zero element $z\in S(\partial M ,\mathbb{Z} [A^{\pm 1}])$ such that $z.f =0$. It implies in our case the existence of a non-zero element $z$ in $S( \partial E_{K'} ,\mathbb{Q} [A^{\pm 1}])$ such that $z.f=0$.\\
Since $S( \partial E_{K'} ,\mathbb{Q} [A^{\pm 1}])\simeq S(\mathbb{T}^2 , \mathbb{Q} [A^{\pm 1} ])$, we can express $z$ in the Frohman-Gelca basis and get :
$$
\left(\mathlarger{\mathlarger{\sum}}\limits_{(\alpha, \beta) \in \mathcal{P}^f \cap \mathbb{Z}^2 } c_{\alpha ,\beta }^f (\alpha ,\beta )_T \right) \cdot f =0
$$ 
where $\mathcal{P}^f$ is a polygon with vertices in $\mathbb{Z}^2$. Because $(\alpha ,\beta)_T = (-\alpha , - \beta )_T$ in the Frohman-Gelca basis, this relation can be chosen such that $(-\mathcal{P}^f)=\mathcal{P}^f$ and $c_{\alpha,\beta}^f=c_{-\alpha ,-\beta}^f \in \mathbb{Q} [A^{\pm 1}]$ for $(\alpha ,\beta )\in \mathcal{P}^f \cap \mathbb{Z}^2$. Moreover, $c_{\alpha , \beta }^f \neq 0$ for $(\alpha ,\beta ) \in \partial \mathcal{P}^f \cap \mathbb{Z}^2$.
\end{proof}

Let $K$ be a knot verifying condition (\ref{Condition}) and let $F$ be a set of generators for $S(E_K , \mathbb{Q} [A^{\pm 1}] )$ over $S(\partial E_K , \mathbb{Q} [A^{\pm 1}] )$. For each $f\in F$, let $\mathcal{P}^{f}$ and $c_{\alpha, \beta}^{f}$ be given by Lemma \ref{Lemma:Polygone} and let
$$
U := \mathlarger{\prod\limits_{f\in F}}\  \mathlarger{\prod\limits_{(\alpha,\beta)\in \partial\mathcal{P}^{f}\cap \mathbb{Z}^2}} c_{\alpha ,\beta}^{f}
$$

Now that we have introduced all the elements we needed, we can state the Corollary \ref{Cor:Tame} more precisely :

\begin{prop}\label{Prop:Tame}
For all slopes $r$ that are not slopes of any of the polygons $\mathcal{P}^f$, $S(E_K (r) ,R_U )$ is finitely generated over $R_U = \mathbb{Q} [A^{\pm 1} ] [U^{-1}]$.
\end{prop}

\begin{proof}

To start with, since $S(E_K , R_U) = S(E_K , \mathbb{Q} [ A^{\pm 1 }]) \otimes R_U$, $F$ also generates $S(E_K , R_U)$ over $S(\partial E_K , R_U )$.\\
Since every element of $S(E_K (r) , R_U)$ can be isotoped into $E_K$, to show that $S(E_K (r) , R_U)$ is finitely generated over $R_U$, it suffices to show that $S(E_K , R_U)$ if finitely generated over $R_U$ as a subspace of $S(E_K (r) , R_U )$. This can be done by showing that $S(\partial E_K , R_U) \cdot f \subset S(E_K (r) , R_U)$ is finitely generated over $R_U$ for every $f\in F$. In the following, we fix a generator $f\in F$.

First, we can multiply the relation of Lemma \ref{Lemma:Polygone} on the left with an element $(\mu , \nu)_T \in S(\partial E_K , R_U)$. Then, using the product-to-sum formula, we obtain : 
\begin{align*}
0 &= (\mu , \nu)_T \star \left(\mathlarger{\mathlarger{\sum}}\limits_{(\alpha, \beta) \in \mathcal{P}^f \cap \mathbb{Z}^2 } c_{\alpha ,\beta }^f (\alpha ,\beta )_T \right) \cdot f 
\\&= \left(\mathlarger{\mathlarger{\sum}}\limits_{(\alpha, \beta) \in \mathcal{P}^f \cap \mathbb{Z}^2 }  A^{\mu \beta - \nu \alpha } c_{\alpha ,\beta }^f (\alpha + \mu ,\beta + \nu )_T  + \mathlarger{\mathlarger{\sum}}\limits_{(\alpha, \beta) \in \mathcal{P}^f \cap \mathbb{Z}^2 }  A^{-\mu \beta + \nu \alpha } c_{\alpha ,\beta }^f (\alpha - \mu ,\beta - \nu )_T   \right) \cdot f   \\&
= \left(\mathlarger{\mathlarger{\sum}}\limits_{(\alpha, \beta) \in \mathcal{P}^f \cap \mathbb{Z}^2 }  A^{\mu \beta - \nu \alpha } c_{\alpha ,\beta }^f (\alpha + \mu ,\beta + \nu )_T  + \mathlarger{\mathlarger{\sum}}\limits_{(\alpha, \beta) \in (-\mathcal{P}^f) \cap \mathbb{Z}^2 }  A^{\mu \beta - \nu \alpha } c_{-\alpha ,-\beta }^f (-\alpha - \mu , - \beta - \nu )_T   \right) \cdot f 
\end{align*}
Since $(-\alpha - \mu , -\beta -\nu )_T = (\alpha + \mu , \beta + \nu )_T$, $c_{-\alpha,-\beta}^f=c_{\alpha ,\beta}^f$ and $(-\mathcal{P}^f )= \mathcal{P}^f$, the last line becomes:
\begin{equation}\label{1}\tag{1}
\left(\mathlarger{\mathlarger{\sum}}\limits_{(\alpha, \beta) \in \mathcal{P}^f \cap \mathbb{Z}^2 } 2 A^{\mu \beta - \nu \alpha } c_{\alpha ,\beta }^f (\alpha + \mu ,\beta + \nu )_T \right) \cdot f = 0
\end{equation}

Keeping in mind Relation (\ref{1}), we get a second relation from the surgery : performing a surgery of slope $r=\dfrac{q}{p}$ (with $gcd(p,q)=1$) on $K$ makes the curve $\gamma_{p,q}=(p,q)_T$ trivial. Thus, $(p,q)_T \cdot f = (-A^2 -A^{-2}) \cdot f$ in $S(E_K (r ),  R_U )$. We then multiply by $(\mu , \nu )_T$ on the right and use the product to sum formula to deduce more relations, we obtain :
\begin{equation}\label{2}\tag{2}
\left( A^{p\nu-q\mu} (p+ \mu ,q+ \nu )_T + (-A^2 - A^{-2}) (\mu ,\nu)_T + A^{q\mu-p\nu} (p-\mu ,q- \nu)_T \right) \cdot f = 0
\end{equation}

To show that $S(\partial E_K , R_U) \cdot f $ is finitely generated over $R_U$, choose two morphisms $\lambda ,\epsilon : \mathbb{Z}^2 \rightarrow \mathbb{Z}$ such that $\lambda \neq 0$, $\lambda (p,q) = 0$ and $\epsilon (p,q) = 1$.\\
Since $\frac{q}{p}$ is not a slope of $\mathcal{P}^f$, $\lambda$ has a unique maximum $M$ and a unique minimum $-M$ on $\mathcal{P}^f$.\\
Let $(x,y)$ be such that $\lambda (x,y)\geq M$ and let $(a,b)$ realize the maximum for $\lambda$ over $\mathcal{P}^f$. Relation (\ref{1}) with $(\mu,\nu) := (x-a,y-b)$ gets $(x,y) = (a + \mu , b + \nu)$ to be the unique maximum for $\lambda$ between all the vertices involved in the relation. Since $c_{(a,b)}^f \neq 0$, this gives an expression of $(x,y)_T \cdot f$ as a linear combination of elements with lesser images by $\lambda$. Note that we need to inverse the coefficients $2A^{\mu \beta - \nu \alpha} c_{\alpha,\beta}^f$ for every vertices $(\alpha ,\beta )$ of $\mathcal{P}^f$, which may not be possible in $\mathbb{Q} [A^{\pm 1}]$ but is possible in $R_U$.\\
By doing this also for the unique minimum $-M$ of $\mathcal{P}^f$, we find that $S(\partial E_K , R_U) \cdot f$ is spanned by elements $(x,y)_T \cdot f$ such that $-M\leq \lambda (x,y) \leq M$.

Similarly, since $A^{p\nu -q\mu}$ is invertible and because $(p-\mu, q-\nu )_T = (\mu -p,\nu -q)_T$, relation (\ref{2}) expresses $(\mu + p , \nu + q)_T \cdot f$ (resp. $(\mu - p , \nu - q)_T \cdot f$) as a linear combination of elements with same image by $\lambda$ but lesser (resp. greater) image by $\epsilon$.

In the end, $S(\partial E_K , R_U) \cdot f$ is spanned by elements $(x,y)_T \cdot f$ such that $-M\leq \lambda (x,y) \leq M$ and $0\leq \epsilon (x,y) \leq 1$ which have coordinates in the intersection of two non-parallel bands of $\mathbb{Z}^2$ and thus form a finite set.
\end{proof}
\begin{remark}
Fixing the slope and the associated $\lambda$ (if possible), the choice of $U$ can be reduced to the product of coefficients $c_{\alpha ,\beta }^f$ with $(\alpha,\beta)$ realising the maximum and the minimum of $\lambda$ on $\mathcal{P}^f$ for each generator $f$.
\end{remark}
\subsection{The proof of Theorem \ref{Thm:Main}}\label{SubSec:Proof}
We adapt the method of \cite[Theorem 3.1]{DKS25} under condition (\ref{Condition}) :
\begin{proof}[Proof of Theorem \ref{Thm:Main}]
By Proposition \ref{Prop:Tame}, there exists a polynomial $U\in\mathbb{Q} [A^{\pm 1}]$ for which $S( E_K (r) ,R_U)$ is finitely generated over $R_U=\mathbb{Q} [A^{\pm 1}] [U^{-1}]$.\\
The ring $R_U$ is a PID as a localization of a PID (see \cite[Prop. 3.11]{AtiMac} for instance). Then, having $S(E_K (r), R_U )$ finitely generated over $R_U$ gives it a decomposition as
$$
S(E_K (r), R_U ) = F\ \underset{i}{\bigoplus}\ \faktor{R_U}{q_i^{s_i}}
$$
where $F$ is a free $R_U$-module and the direct sum is finite over certain powers of certain irreducibles $q_i \in R_U$, $q_i\neq 1$, possibly repeating themselves.\\
It follows that $\dim_{\mathbb{Q} (A)} (S(E_K (r))) = rk_{R_U} (F)$ :
$$
S (E_K (r)) = S(E_K (r), R_U ) \otimes \mathbb{Q} (A) \simeq \left(\mathbb{Q} (A)\right)^{rk_{R_U} (F)} 
$$
On the other hand, let $\zeta$ be a primitive $2N$-root of unity, such that $\zeta$ is not a root of any $q_i$ nor a root of $U$. Thus,  $\faktor{R_U}{q_i^{s_i}} \underset{A=\zeta}{\otimes} \mathbb{C} = 0$ and :
$$
S_\zeta (E_K (r)) = S(E_K (r), R_U ) \underset{A=\zeta}{\otimes} \mathbb{C} \simeq \mathbb{C}^{rk_{R_U} (F)} 
$$
Thus, $\dim_{\mathbb{Q} (A)} (S(E_K (r))) = rk_{R_U} (F) = \dim_{\mathbb{C}} (S_\zeta (E_K (r)))$.
\end{proof}

\section{Dimension of \texorpdfstring{$S_\zeta (M)$}{}}\label{Sec:CharVar}
Under Condition (\ref{Condition}), as a 3-manifold with finite character variety (Corollary \ref{Cor:FiniteX}), knowing $S_\zeta (E_K (r) )$ only requires to understand the localized skein modules $S_{\zeta,[\rho ]} (E_K (r))$.

In order to explain the state of the art on localized skein modules, we will use some notions of affine PI algebras :

\subsection{Almost Azumaya algebras}
\begin{definition}
Let $\mathcal{A}$ be a $\mathbb{C}$-algebra.\\
If $\mathcal{A}$ is affine, prime with finite rank over its center, then $\mathcal{A}$ is said to be almost Azumaya.
\end{definition}
In this case (see \cite[III.1.2]{BroGoo}), there is an integer $D$ such that the dimension of $\mathcal{A}\underset{Z(\mathcal{A})}{\otimes} Frac(Z(\mathcal{A}))$ over $Frac(Z(\mathcal{A}))$ is $D^2$. The integer $D$ is called the PI-degree of $\mathcal{A}$.
\begin{definition}
If $\mathcal{A}$ is almost Azumaya, the Azumaya locus is 
$$Azu(\mathcal{A})=\{ \mathfrak{m} \in MaxSpec(Z(\mathcal{A})) , \faktor{\mathcal{A}}{\mathfrak{m} \mathcal{A}} \simeq M_{D} (\mathbb{C} )\}$$
For a finitely generated $\mathcal{A}$-module $\mathcal{K}$, we also define 
$$
Azu'_{\mathcal{A}} (\mathcal{K}) := \{ \mathfrak{m} \in MaxSpec(Z(\mathcal{A}) ), \dim_{\mathbb{C}} (\faktor{\mathcal{K}}{\mathfrak{m}\mathcal{K}}) =\dim_{ Frac(Z(\mathcal{A}))}  (\mathcal{K}\underset{Z(\mathcal{A})}{\otimes} Frac(Z(\mathcal{A}))) \} \
$$
\end{definition}
\begin{prop}\cite[Theorem III.1.7]{BroGoo}\label{prop:AzuOpen}
$Azu(\mathcal{A})$ is Zariski open.
\end{prop}
\begin{prop}\label{prop:Azu'Open}
$Azu'_{\mathcal{A}} (\mathcal{K})$ is Zariski open.
\end{prop}
\begin{proof}
Let $d=\dim_{ Frac(Z(\mathcal{A}))}  (\mathcal{K}\otimes Frac(Z(\mathcal{A})))$.\\
For $\mathfrak{p}\in Spec(Z(\mathcal{A})) $, let $\kappa (\mathfrak{p} ) = \faktor{Z(\mathcal{A})_{\mathfrak{p}}}{\mathfrak{p}Z(\mathcal{A})_{\mathfrak{p}}}$ be the residue field of $\mathfrak{p}$.\\
Let $\Lambda : Spec (Z (\mathcal{A})) \rightarrow \mathbb{N}$ be defined by $\Lambda (\mathfrak (p)) = \dim_{\kappa (\mathfrak{p} )} (\mathcal{K} \otimes \kappa (\mathfrak{p}))$.\\
For $\mathfrak{m}\in MaxSpec (Z(\mathcal{A}))$, we have that $\kappa(\mathfrak{m} ) = \mathbb{C}$ and $\Lambda (\mathfrak{m} ) = \dim_{\mathbb{C}} \faktor{\mathcal{K}}{\mathfrak{m}\mathcal{K}}$. Moreover, since $\mathcal{A}$ is prime, $Z(\mathcal{A} )$ has no zero divisors, then $(0)\in Spec(Z(\mathcal{A}))$ and we have $\kappa ( (0) )= Frac(Z(\mathcal{A}))$ and $\Lambda ( (0) ) = d$.\\
It is known that $\Lambda$ is upper semi-continuous (\cite[Example 12.7.2]{Har}). In particular, for every $\mathfrak{p} \in Spec(Z(\mathcal{A}))$, the set $\{ \mathfrak{p'} \in Spec(Z(\mathcal{A})) , \Lambda (\mathfrak{p'} ) \leq \Lambda(\mathfrak{p} ) \}$ is an open neighborhood of $\mathfrak{p}$.\\
Moreover, due to the form of the usual Zariski basis, $(0)$ is included in every neighborhood of any prime ideal ($(0)$ is called a generic point of $Spec(Z(\mathcal{A}))$).\\
This shows that $d=\Lambda ((0))$ is the minimal dimension for $\{ \kappa (\mathfrak{p} ) \}_{\mathfrak{p}\in Spec (Z(\mathcal{A}))}$.\\
It implies that $Azu'_{\mathcal{A}} (\mathcal{K}) = \{ \mathfrak{m} \in MaxSpec(Z(\mathcal{A})) , \Lambda (\mathfrak{m} ) \leq d \}$ and then, again by the upper semi-continuity of $\Lambda$, $Azu'_{\mathcal{A}} (\mathcal{K}) $ is open.
\end{proof}
\subsection{Reduced skein module}
Following \cite{TehFroKan}, for $M$ be a closed oriented 3-manifold, we consider an Heegaard splitting $M=:H_1 \underset{\Sigma}{\cup} H_2$ of $M$. 

Then, the map $S_{-1} (\Sigma ) \rightarrow S_{-1} (M )$ is surjective and we can consider $Spec (S_{-1} (M))$ as a subspace of $Spec (S_{-1} (\Sigma ))$. 

Also, since $\partial H_i = \Sigma $, we view $S_\zeta (H_i )$ as a $S_\zeta (\Sigma )$-module.

Moreover, by \cite[Theorem 4.1]{FKL}, $Z (S_\zeta (\Sigma ))\simeq S_{-1} (\Sigma )$ through the map given by Theorem \ref{Thm:BohWon}.

\begin{prop}\cite[Theorem 5.1]{FKL}\label{Prop:FKL}
Let $i\in \{ 1,2\}$.\\
The algebra $\mathcal{A} := S_\zeta (\Sigma )$ is almost Azumaya and $\mathcal{K}:= S_\zeta (H_i )$ is a finitely generated $\mathcal{A}$-module.
\end{prop}
We now can talk about $Azu(S_\zeta (\Sigma ))$ and $Azu'_{S_\zeta (\Sigma )} (S_\zeta (H_i))$. Let 
$$D^2:=\dim_{ Frac(Z(\mathcal{A}))}  (\mathcal{A}\underset{Z(\mathcal{A})}{\otimes} Frac(Z(A)))$$
And
$$d= \dim_{ Frac(Z(\mathcal{A}))}  (\mathcal{K} \underset{Z(\mathcal{A})}{\otimes} Frac(Z(\mathcal{A})))$$
We will show in Proposition \ref{prop:HAzu} that $D=d$.
\subsection{Non-central characters}
We are now ready to describe the important results about localized skein modules. The first one describes the localized skein modules of $\Sigma$ at non-central characters.
\begin{prop}\cite[Theorem 1.1.4]{GaJoSa}\cite[Theorem 1.2]{KarKor25}\label{prop:SAzu}
For $[\rho]$ the character of a non-central representation of $\chi (M)$, $\mathfrak{m}_{[\rho]} \in Azu(S_\zeta (\Sigma ))$.
\end{prop}
The second one describes the localized skein modules of the handlebodies $H_i$ at non-central characters.
\begin{prop}\cite[Theorem 12.1]{FKBL}\cite[Lemma 6.5]{KarKor25}\label{prop:HAzu}
For $i\in \{ 1, 2 \}$ and $[\rho]$ the character of a non-central representation of $\chi (M)$, $\mathfrak{m}_{[\rho]} \in Azu'_{S_\zeta (\Sigma )} (S_\zeta (H_i))$. Moreover, $D=d$.
\end{prop}
\begin{proof}
The only point not adressed in the two references is the equality $D=d$, but since $D$ is the dimension of the reduced skein module at points in the Azumaya locus and that, thanks to the dimension given in \cite[Theorem 12.1]{FKBL}, the equality is true at least for irreducible characters, it is true on all $Azu (S_\zeta (\Sigma))$.
\end{proof}
The two latter results will be used through the following :
\begin{theorem}\cite{TehFroKan}\label{Thm:ReprNonCentrale}
Let $[\rho ]$ be a non-central representation of $S(M)$ and let $m_{[\rho ]}$ be the multiplicity of $[\rho ]$, then : $$S_\zeta (M)_{[\rho ]} \simeq  S_{-1} (M)_{[\rho ]} \simeq \mathbb{C}^{m_{[\rho ]}}$$
\end{theorem}
\begin{proof}
Since the theorem of \cite{TehFroKan} only adresses irreducible characters we explain how to use their proof in the general case.\\
The key idea is to notice that the only thing needed in \cite{TehFroKan} about $[\rho ]$ is to verify the hypothesis of \cite[Prop. 3.3]{TehFroKan} with both $(K,A)=(S_{-1} (\Sigma),S_\zeta (\Sigma ))$ and $(K,A)=(S_{-1} (H_i),S_\zeta (H_i))$.\\
This is done by Proposition \ref{prop:SAzu} and Proposition \ref{prop:HAzu} which ensure that every non-central representation $[\rho]$ is in $( Azu(S_\zeta (\Sigma) ) \underset{i\in \{1,2\}}{\bigcap} Azu'_{S_\zeta (\Sigma )} (S_\zeta (H_i)))$ and by Proposition \ref{prop:AzuOpen} and Proposition \ref{prop:Azu'Open} which give the open conditions.\\
The rest of the paper follows by replacing the use of \cite[Prop. 4.2]{TehFroKan} and \cite[Theorem. 4.1]{TehFroKan} in \cite[Prop. 5.4]{TehFroKan} (through \cite[Prop. 3.3]{TehFroKan}).
\end{proof}
\subsection{The total skein module}
Since central characters are isolated and reduced when $X(M)$ is finite, reduced skein modules at central characters are the same as localized skein modules. Then we have the following :
\begin{prop}\cite[Lemma 4.5]{Kor}\label{prop:TrivialRep}
Let $M$ be an oriented closed 3-manifold with finite $X(M)$ and $[\rho ],[\rho '] \in \chi (M)$ be two central characters, then $S_\zeta (M)_{[\rho ]} \simeq S_\zeta (M)_{[\rho' ]}$ 
\end{prop}
For the sake of completedness, we transcribe the proof below.
\begin{proof}
Let $L,L'\in S_\zeta (M)$ be represented by links, let $K_1,\ldots , K_n$ be the components of $L$, and let $r(L,L')=T_N (L) \sqcup L'- \prod\limits_{i=1}^{n} (-tr (\rho (K_i ))) L'$. Then, $S_{\zeta,[\rho ] } (M)$ is the quotient of $S_\zeta (M)$ by all the possible relations of the form $r(L,L')$ (and likewise for $S_{\zeta,[\rho']} (M)$).\\
Using the fact that the skein relations are $H^1 (M,\faktor{\mathbb{Z}}{2\mathbb{Z}})$-homogeneous, for $\omega \in H^1 (M,\faktor{\mathbb{Z}}{2\mathbb{Z}})$, the automorphism $f_\omega : S_\zeta (M) \rightarrow S_\zeta (M)$ determined by $f_\omega (L) = (-1)^{\sum\omega (K_i)} L$, for $L$ represented by link of components $K_1 ,\ldots K_n$, is well defined.\\
Since $\rho$ and $\rho'$ are both central representation, there exists $\omega \in H^1 (M,\faktor{\mathbb{Z}}{2\mathbb{Z}})$ such that for every knot $K$, $(-1)^{\omega (K)} tr(\rho (K)) =tr(\rho' (K)) $.\\
Since $T_N (-X)=-T_N (X)$, the automorphism $f_\omega$ descends to an isomorphism $ S_{\zeta,[\rho ] } (M) \simeq S_{\zeta,[\rho' ]} (M) $. Using the fact that $[\rho ]$ and $[\rho ']$ are reduced concludes the proof.
\end{proof}
We conclude with the following :
\begin{prop}\label{prop:Decomp}
Let $M$ be a closed oriented 3-manifold such that $\chi (M)$ if finite. For $[\rho ] \in \chi (M )$, let $n_{[\rho ]}$ be the multiplicity of $[\rho ]$. Let $\chi_0 \subset \chi( M)$ be the set of central characters and let $\mathbbm{1}$ be the trivial representation.\\
Then, for all primitive $2N$-roots of unity $\zeta$ with $N$ odd,
$$S_\zeta (M ) = (S_\zeta (M )_{[\mathbbm{1}]} )^{\vert \chi_0 \vert } \underset{[\rho ]\in \chi(M )\setminus \chi_0}{\bigoplus} \mathbb{C}^{n_{[\rho ]}}$$
\end{prop}
\begin{corollary}
Let $K$ verifying condition (\ref{Condition}).\\
Proposition \ref{prop:Decomp} applies to $E_K (r)$ for almost all slopes $r\in \mathbb{Q}$ and almost all primitive $2N$-roots of unity $\zeta$ with $N$ odd.
\end{corollary}
\begin{proof}
The right part of this decomposition is coming from Theorem \ref{Thm:ReprNonCentrale} and the left part is from Proposition \ref{prop:TrivialRep}.
\end{proof}
Unfortunately, we don't know $S_\zeta (M )_{[\mathbbm{1}]}$ yet, but we make the following conjecture :
\begin{Conj}
Let $M$ be an oriented closed 3-manifold and $\zeta$ be a primitive $2N$-root of unity with $N$ odd. If $X (M)$ is finite, then 
$$S_\zeta (M )_{[\mathbbm{1}]} \simeq \mathbb{C}$$
Implying that, for $K$ verifying condition (\ref{Condition}) and for almost all $r\in\mathbb{Q}\cup\{\infty\}$, we have that $S_\zeta (E_K (r) ) \simeq \mathbb{C}^{n}$, where $n$ is the number of characters of $X(M)$ counted with multiplicity.
\end{Conj}
\subsection{Comparison with \texorpdfstring{\cite{DKS25}}{}}\label{Subsec:Comp}
Recall that $\vert X(M) \vert$ is the number of points of $X(M)$ counted with multiplicity. Let $\eta$ be the counting without multiplicity.\\
First we prove Remark \ref{Rk:NotReduced} :

In their work, \cite{DKS25} shows two inequalities : $\eta \leq \dim_{\mathbb{Q} (A)} S(M)$ and $\dim_{\mathbb{Q} (A)} S(M) \leq \dim_{\mathbb{C}} \mathbb{C} [\chi (M)] = \vert X(M) \vert$. But, using Proposition \ref{prop:Decomp}, we have the inequality $\vert X(M) \vert \leq \dim_{\mathbb{Q} (A)} S(M) $. This suffices to get their result without the first equation and then without the reduced (equivalently $\eta =\vert X(M) \vert$) assumption.

Another remark is that it is tempting to try to prove the inequality  $\dim_{\mathbb{Q} (A)} S(M) \leq \vert X(M) \vert$ following the same path as \cite{DKS25} in our setting, through the decomposition of Section \ref{SubSec:Proof} :
$$
S(E_K (r), R_U ) = F\ \underset{i}{\bigoplus}\ \faktor{R_U}{q_i^{s_i}}
$$
However, here, $-1$ could be a root of $U$. In this case, we have $R_U \underset{A=-1}{\bigotimes} \mathbb{C} = 0$ and we cannot recover the dimension of $S_{-1} (E_K (r))$ by this decomposition. 
\bibliographystyle{hamsalpha}
\bibliography{biblio}
\texttt{Institut de Mathématiques de Bourgogne, UMR 5584 CNRS, Université de Bourgogne,\\ F-21000, Dijon, France}\\
\textit{Email address : edwin.kitaeff@u-bourgogne.fr}
\end{document}